\documentclass[11pt]{amsart}

\title[Powell's conjecture on the Goeritz group]{A proof of Powell's conjecture on the Goeritz group of $S^3$}
\author{Daiki Iguchi}

\address{Higashihiroshima, Japan}
\email{diguchi00@gmail.com}


\usepackage{amsfonts,amsmath,amsthm,amssymb,amscd,mathrsfs}
\usepackage{latexsym}
\usepackage{graphicx}
\usepackage{psfrag}
\usepackage{color}
\usepackage[abs]{overpic}
\usepackage{caption}
\usepackage[all]{xy}
\usepackage{keytheorems}
\usepackage{enumitem}
\usepackage[colorlinks=true, allcolors=blue]{hyperref}

\newkeytheorem{theorem}[parent=section]
\newkeytheorem{lemma}[sibling=theorem]
\newkeytheorem{corollary}[sibling=theorem]
\newkeytheorem{proposition}[sibling=theorem]
\newkeytheorem{definition}[sibling=theorem]
\newkeytheorem{claim*}[name=Claim,numbered=false]
\newkeytheorem{conjecture}[sibling=theorem]

\newcommand{\interior}{\operatorname{int}}

\newcommand{\bd}[1]{\partial#1}

\newcommand{\Diff}{\operatorname{Diff}}

\newcommand{\MCG}{\operatorname{MCG}}

\newcommand{\lft}[1]{\widetilde{#1}}

\newcommand{\id}{\mathrm{id}}

\newcommand{\sSph}{\mathcal{S}}
\newcommand{\sBall}{\mathcal{B}}
\newcommand{\sCyl}{\mathcal{C}}
\newcommand{\grph}{\mathscr{G}}
\newcommand{\redg}{\mathscr{G}_{\mathrm{red}}}

\begin{document}

\begin{abstract}
     For a genus $g$ Heegaard splitting of the $3$-sphere, the Goeritz group is defined to be the group of isotopy classes of diffeomorphisms of the $3$-sphere that preserve the splitting setwise. In this paper, we prove the following conjecture proposed by Powell: For every $g \ge 3$, the Goeritz group of a genus $g$ Heegaard splitting is generated by four specific elements. Our proof relies crucially on the fact that a Heegaard surface of the $3$-sphere is topologically minimal, that is, its disk complex has nontrivial homotopy group in some dimension. Along the way, we also give a new proof of the fact that a genus $g$ Heegaard surface of the $3$-sphere has topological index $2g-1$.  
\end{abstract}

\maketitle

\section{Introduction}\label{sec:Introduction}

Let $M$ be a closed orientable $3$-manifold. A \emph{genus $g$ Heegaard splitting} of $M$ is a decomposition of $M$ into two handlebodies $A$ and $B$ along a genus $g$ surface $T$. We denote by $\Diff(M,T)$ the group of those diffeomorphisms $\tau:M \to M$ such that $\tau(A)=A$. The \emph{mapping class group} $\MCG(M,T)$ of the Heegaard splitting $(M,T)$ is the group of path components of $\Diff(M,T)$, and the \emph{Goeritz group} $G(M,T)$ is the kernel of the map $\MCG(M,T) \to \MCG(M)$.  

Alternatively, $G(M,T)$ can be viewed as the quotient of the fundamental group of the space $\Diff(M)/\Diff(M,T)$ of Heegaard splittings \cite[Theorem~1]{Johnson-McCullough}: any element of $G(M,T)$ is represented by an isotopy $T_\theta$, $\theta \in [0,2\pi]$, that starts with $T$ and comes back to the initial position. For example, consider a genus $g$ Heegaard splitting $(S^3,T)$ of the $3$-sphere. For $g \ge 3$, Powell proposed five isotopies that represent elements of $G(S^3,T)$ \cite{Powell}. Four of the five isotopies are depicted in Figures~\ref{fig:Powell_moves} and \ref{fig:standard_twist}. (One of the isotopies is known to be redundant \cite{Scharlemann20}.) 

Although it was once believed to be proved that these isotopies represent generators of $G(S^3,T)$, Scharlemann \cite{Scharlemann04} pointed out a subtle gap in the proof in \cite{Powell}. So the problem remains open for $g >3$ until now, while the genus $3$ case has recently been resolved \cite{Freedman-Scharlemann,Cho-et-al}. The main purpose of the present paper is to give a new correct proof of Powell's conjecture: 

\begin{theorem}[The Powell conjecture]\label{thm:Powell}
    For $g \ge 3$, the Goeritz group $G(S^3,T)$ of genus $g$ Heegaard splitting of $S^3$ is generated by the four elements $D_\omega$, $D_\eta$, $D_{\eta_{12}}$ and $D_\theta$ shown in Figures~\ref{fig:Powell_moves} and \ref{fig:standard_twist}. 
\end{theorem}

A remarkable feature of a Heegaard surface of $S^3$ is that it is a topologically minimal surface. The notion of a topologically minimal surface was introduced by Bachman \cite{Bachman}. For example, if a surface $F \subset M$ is incompressible or strongly irreducible, then $F$ is a topologically minimal surface of index $0$ or $1$, respectively. By \cite{Appel,Campisi-Torres}, it is known that a genus $g$ Heegaard surface of $S^3$ has index $2g-1$. We also give an alternative proof of this fact in Section~\ref{sec:bound4index}, as a byproduct of the proof of Theorem~\ref{thm:Powell}. 

It is always natural to ask whether a result on incompressible or strongly irreducible surfaces can be generalized to surfaces of higher index. See \cite{Bachman} for examples. In relation to the Goeritz groups, the author \cite{Iguchi} proved that any strongly irreducible Heegaard splitting of an irreducible atoroidal $3$-manifold has the finitely generated Goeritz group, based on ideas by \cite{Colding-et-al18}. The same also holds for Heegaard splittings of higher index under some additional assumption. These facts suggest that a correct approach to Theorem~\ref{thm:Powell} would rely on topological minimality of $T$, and this is a motivation of the present paper.  

This paper is organized as follows. Section~\ref{sec:outline} is an outline of the proof of Theorem~\ref{thm:Powell}. In Section~\ref{sec:diskcomplex}, we recall the definitions of the disk complex and its variant, and discuss their basic properties. After reviewing sweepout theory in Section~\ref{sec:sweepout}, we will see that the argument in \cite{Bachman} can be adapted to Heegaard splittings of $S^3$ in Section~\ref{sec:balanced_sphere}. At this point, we can prove some interesting results. For example, we will prove an analogue of Theorem~3.2 in \cite{Bachman}. As a corollary of this result, we also give a lower bound on the index of a Heegaard splitting of $S^3$. These results are proved in Section~\ref{sec:bound4index}. Key ideas of the proof of Theorem~\ref{thm:Powell} are presented in Section~\ref{sec:parallelization}, after establishing several technical results in Sections~\ref{sec:sliding}. Finally, the proof of Theorem~\ref{thm:Powell} is completed in Section~\ref{sec:finalstep}.  

I would like to thank Martin Scharlemann for his valuable comments. 

\section{Outline of the proof}\label{sec:outline}

We start with some setups. Let $S$ be a $2$-sphere in $S^3$. Fix a finite graph $K \subset S$ that is a bouquet of circles $e^1,e^2,\ldots,e^g$ that bound disks $\Delta^1,\Delta^2,\ldots,\Delta^g \subset S$, respectively, such that $\interior \Delta^i \cap \interior \Delta^j=\emptyset$ for $i \neq j$. We view the Heegaard surface $T$ as a boundary of a neighborhood of $K$, as shown in Figure~\ref{fig:standard_spine}. We denote by $A$ the handlebody with $\bd{A}=T$ that contains $K$, and by $B$ the other handlebody. Note that by definition $K$ is a \emph{spine} of $A$, that is, $A \setminus K$ is homeomorphic to $T \times (0,1]$. 

\begin{figure}
    \begin{overpic}[scale=.5]{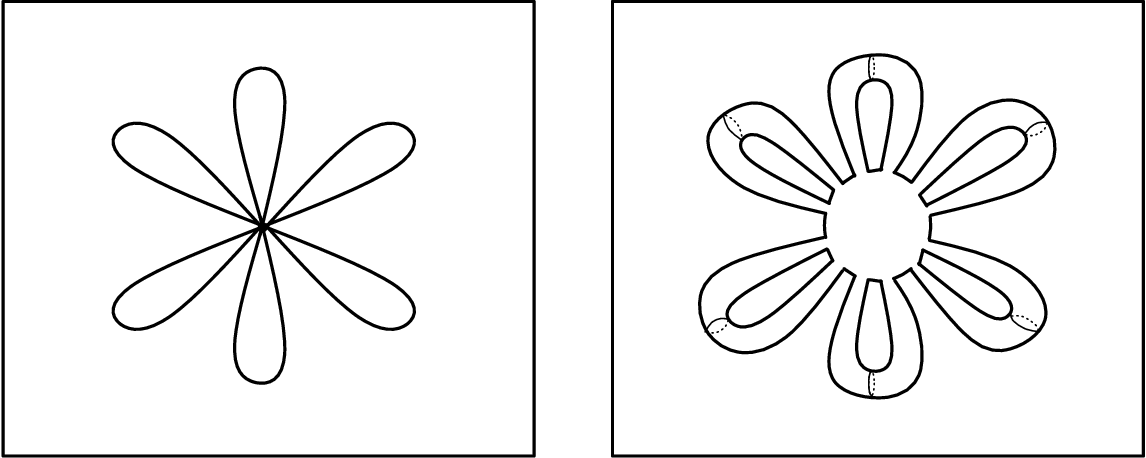}
        \put(5,5){$S$}
        \put(100,70){$e^1$}
        \put(70,90){$e^2$}
        \put(35,85){$e^3$}
        \put(30,45){$e^4$}
        \put(45,20){$e^5$}
        \put(95,45){$e^6$}
        \put(260,70){$T$}
    \end{overpic}    
    \caption{Left: a standard spine $K$. Right: the Heegaard surface $T$ can be viewed as the boundary of a regular neighborhood of $K$.}\label{fig:standard_spine}
\end{figure}
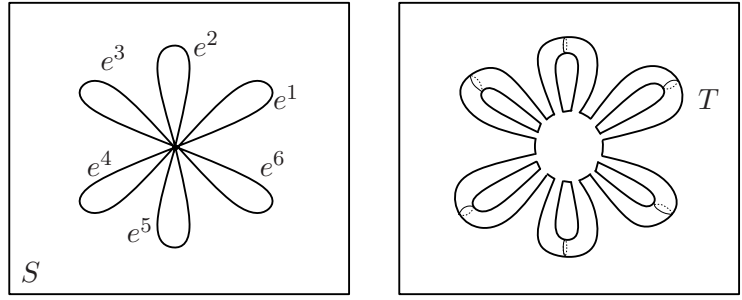

Let $T_\theta$, $\theta \in [0,2\pi]$, be an isotopy of a Heegaard surface with $T_0=T_{2\pi}=T$. We say $T_\theta$ is a \emph{Powell move} if $T_\theta$ is equivalent to a composition of moves shown in Figures~\ref{fig:Powell_moves} and \ref{fig:standard_twist}. 

\begin{figure}
    \begin{overpic}[scale=.6]{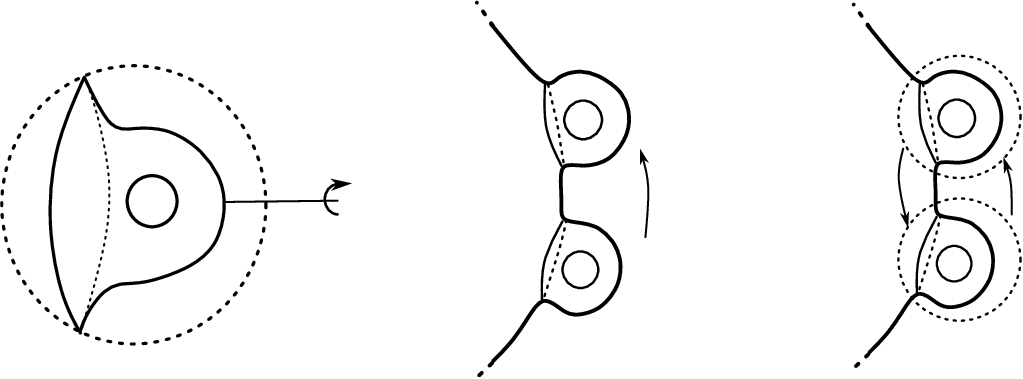}
        \put(103,47){$\pi$}
        \put(200,47){$2\pi/g$}
    \end{overpic}    
    \caption{$D_\omega$ flips the first $1$-handle (left), $D_\eta$ rotates $T$ along the $z$-axis (middle), and $D_{\eta_{12}}$ exchanges the first and the second $1$-handles (right).}\label{fig:Powell_moves}
\end{figure}

\begin{figure}
    \begin{overpic}[scale=.5]{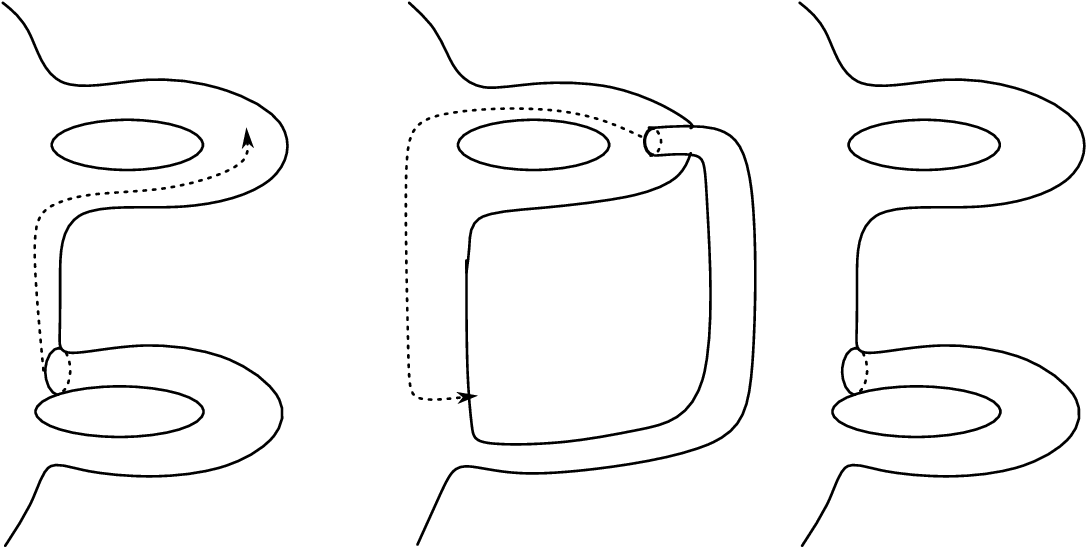}
    \end{overpic}    
    \caption{$D_\theta$ slides the first $1$-handle over the second.}\label{fig:standard_twist}
\end{figure}

It suffices to prove Theorem~\ref{thm:Powell} that any isotopy $T_\theta$ is a Powell move. To state a key step of the proof, we need a few definitions. Let $A_\theta$ be the isotopy of $A$ corresponding to $T_\theta$. 

\begin{definition}
    We will say $T_\theta$ is \emph{supposed by a family of spines} $K_\theta$ if for every $\theta$, $K_\theta$ is a spine of $A_\theta$.
\end{definition}

\begin{figure}
    \begin{overpic}[scale=.5]{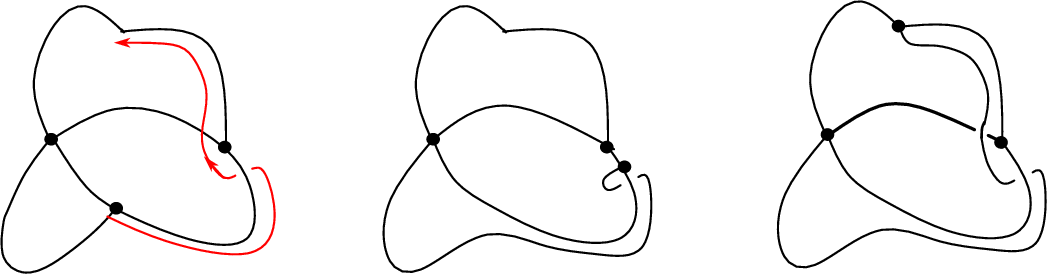}
    \end{overpic}    
    \caption{Sliding an edge along an arc.}
    \label{fig:edgeslide}
\end{figure}

An \emph{edge slide} is an operation on a spine $K'$ that sidles one of the two ends of an edge of $K'$ along an arc in a regular neighborhood of $K'$ (Figure~\ref{fig:edgeslide}). See, e.g., \cite{Scharlemann-et-al} for the precise definition. 

\begin{definition}
    An edge slide of a spine will be called an \emph{$S$-slide} if each intermediate spine lies in $S$.
\end{definition}

A key feature of $S$-slides is that the meridian disk of the edge that is slid is necessarily primitive, see Section~\ref{sec:finalstep}.  

As a key step, we will prove the following theorem: 

\getkeytheorem{planar}

To deduce Theorem~\ref{thm:Powell} from Theorem~\ref{thm:planargraph}, we will follow a strategy proposed by Freedman and Scharlemann in \cite{Freedman-Scharlemann}. Let $a^1$ be a cocore of the first $1$-handle of $A$. By virtue of Theorem~\ref{thm:planargraph}, there are finitely many points $0=\theta_0<\theta_1<\cdots<\theta_{n-1}<\theta_n=2\pi$ and for each $\theta_i$ we can choose a primitive disk $a_{\theta_i} \subset A_{\theta_i}$ that is a cocore of a $1$-handle corresponding to an edge of $K_{\theta_i} \subset S$. To each $\theta_i$ associate a map 
\begin{eqnarray*}
    h_{\theta_i}:(S^3,T_{\theta_i},a_{\theta_i}) \to (S^3,T,a^1).
\end{eqnarray*}
By an inductive argument, we can show that such $h_{\theta_i}$ is unique up to Powell move. Moreover, $h_{\theta_i}\tau_{\theta_i}$ is identical to $h_{\theta_{i+1}}\tau_{\theta_{i+1}}$ up to Powell move. This implies that $\tau_{2\pi}$ is identical to $\id_{S^3}$ up to Powell move, as desired. We will explain more details in Section~\ref{sec:finalstep}. 

Now Theorem~\ref{thm:Powell} is reduced to Theorem~\ref{thm:planargraph}. The proof of Theorem~\ref{thm:planargraph} goes as follows. 

We first construct a sweepout $\{T_t\}$ parametrized by points in the $d$-ball $B^d$, based on ideas in \cite{Bachman}. As usual, we would like to take $\{T_t\}$ to be ``nontrivial'' so that useful information can be extracted from it. The existence of such a sweepout is guaranteed by the following fact. 

\getkeytheorem{index}

This implies that there is a homotopically nontrivial map $\psi_0:S^{d-1} \to \Gamma(T)$ for some $d \le 2g-1$. Roughly speaking, $\{T_t\}$ is constructed by mimicking the map $\psi_0$. 

Fix a height function $f:S^3 \to [-1,1]$ with $f^{-1}(0)=S$, and consider the second sweepout $\{S_s\}$ given by $S_s:=f^{-1}(s)$. Then, we will analyze the intersection pattern between $T_t$ and $S_s$. Using a similar argument to the proof of Theorem~3.2 of \cite{Bachman}, we can prove that $T$ can be isotoped into a kind of normal form:   

\getkeytheorem{normal_position}

It also follows from Theorem~\ref{thm:normal_position} and Lemma~\ref{lem:well-defined_index} that $T$ has topological index $d=2g-1$. These results will be proved in Section~\ref{sec:bound4index}. 

Next, we generalize Theorem~\ref{thm:normal_position} by introducing a new parameter $\theta$. Let $T_\theta$ be an isotopy that represents an element of $G(S^3,T)$. Consider a sweepout given by 
\begin{eqnarray*}
    T_{t\theta}:=\tau_\theta(T_t),
\end{eqnarray*}
where $\tau_\theta:S^3 \to S^3$ is an ambient isotopy that covers $T_\theta$. 

Roughly speaking, the point here is that $T_\theta$ can be isotoped simultaneously for $\theta$ to a position described in Theorem~\ref{thm:normal_position}. This implies that $T_\theta$ can be made ``nearly parallel'' to $S_s$. See Section~\ref{sec:sliding} for the definition. These ideas can be summarized as follows:  

\getkeytheorem{parallelism}

The proof of this theorem is inspired by \cite{Bachman} and \cite{Iguchi}, where the topological minimality of $T$ (Lemma~\ref{lem:well-defined_index}) is essentially used. 

By virtue of Theorem~\ref{thm:parallelization}, we can find a family of spines $K_\theta \subset S$ for $\theta \in [2\pi/3,4\pi/3]$. On the other hand, the isotopy that corresponds to a radial segment in $B^d$ is relatively easy to understand. We study such an isotopy in Section~\ref{sec:sliding} and show that it is supported by a family of spines arising from a sequence of $S$-slides and isotopies. Theorem~\ref{thm:planargraph} now follows from these things, which completes the outline of the proof of Theorem~\ref{thm:Powell}. 

\section{The homotopy index of the disk complex}\label{sec:diskcomplex}

We first recall several definitions from \cite{Bachman}. Let $T$ be a closed separating surface in an irreducible $3$-manifold $M$. The \emph{disk complex} $\Gamma(T)$ is the simplicial complex whose vertices are isotopy classes of compressing disks for $T$, and whose $k$-simplexes are $(k+1)$-tupples of vertices that admit pairwise disjoint representatives. If $\Gamma(T) \neq \emptyset$, the \emph{homotopy index} of $\Gamma(T)$ is defined to be the minimum number $d$ such that $\pi_{d-1}(\Gamma(T)) \neq 1$. If $\Gamma(T)=\emptyset$, define the homotopy index to be $0$. The surface $T$ is said to be \emph{topologically minimal} if $\Gamma(T)$ has well-defined index, and the \emph{topological index} of $T$ is defined to be the homotopy index of $\Gamma(T)$. 

Let $S \subset M$ be another surface. (We do not necessarily assume that $T$ intersects $S$ transversely.) Suppose that $D$ is a compressing disk for $T$. A disk $D'$ is called a \emph{shadow} of $D$ if $\bd{D'}=\bd{D}$, $D' \cap S=\emptyset$ and $\interior D'$ intersects $T$ in loops that are inessential in $T$. Define $\Gamma_S(T)$ to be the subcomplex of $\Gamma(T)$ spanned by those vertices that admit shadows.  

Next, suppose that $T$ is a genus $g$ Heegaard surface of $S^3$ and $S \subset S^3$ is a $2$-sphere. We study what information about $T$ can be obtained by adapting the notion introduced above. To this end, we fix a height function $f:S^3 \to [-1,1]$ with $S=f^{-1}(0)$. Write $S_s:=f^{-1}(s)$, $X^-_s:=f^{-1}([0,s])$ and $X^+_s:=f^{-1}([s,1])$. Then, $\{S_s\}$ is a sweepout of $S^3$ by $2$-spheres. Following \cite{Johnson10}, we say $T$ is \emph{mostly above} (resp. \emph{mostly below}) $S=S_s$ if $T \cap X^-_s$ (resp. $T \cap X^+_s$) is contained in a disk in $T$. 

\begin{definition}
    We will say $T$ is in a \emph{balanced position} with respect to $S$ if $T$ is neither mostly above nor mostly below $S$. A $2$-sphere $S$ is called a \emph{balancing sphere} for $T$. 
\end{definition}

The following proposition says that $T$ can intersect $S$ in some restrictive ways in terms of $\Gamma_S(T)$: 

\begin{proposition}\label{prop:empty-contractible}
    Suppose that $T$ is in a balanced position with respect to $S$. Then, one of the following two holds. 
    \begin{enumerate}
        \item $\Gamma_S(T)=\emptyset$, and hence $\Gamma_S(T)$ has homotopy index $0$. 
        \item $\Gamma_S(T)$ is contractible. 
    \end{enumerate}
\end{proposition}

\begin{proof}
    Suppose that $\Gamma_S(T)$ is not empty. It suffices to show that $\Gamma_S(T)$ is contractible. Let $N_\epsilon \cong S \times [-\epsilon,\epsilon]$ be a small product neighborhood of $S$. We may assume $\bd{N_{\epsilon}}$ intersects $T$ transversely. 
    
    By assumption, there exists a disk $D_1$ with $\bd{D_1} \subset T$ that is a shadow of some compressing disk for $T$. Without loss of generality, we assume that $D_1$ lies in a $3$-ball below $S \times \{-\epsilon\}$. 

    Note that $T$ is not mostly below $S \times \{-\epsilon\}$. Indeed, if this was not the case, $T$ would be mostly below $S$ as well, which is impossible by assumption. This, together with the fact that $D_1$ is entirely below $S \times \{-\epsilon\}$, implies that there exists a loop in $S \times \{-\epsilon\} \cap T$ that is essential in $T$. Let $c$ be an innermost one among all such loops. Then, $c$ bounds a disk $D_0$ in $S \times \{-\epsilon\}$ that is a shadow of some compressing disk for $T$. 

    Let $D$ be a disk that is a shadow of some compressing disk for $T$. Then $\bd{D}$ can be pushed off $N_\epsilon \cap T$ by an isotopy. In particular, $D$ is disjoint from $D_0$ after this isotopy. So, $D_0$ is connected to every vertex of $\Gamma_S(T)$ by an edge. Since $\Gamma(T)$ and hence $\Gamma_S(T)$ are flag by definition, this implies that $\Gamma_S(T)$ is contractible. This proves the proposition. 
\end{proof}

In the remainder of this section, we study a consequence of the first possibility (1) in Proposition~\ref{prop:empty-contractible}. (A consequence of the second possibility will be discussed in Section~\ref{sec:balanced_sphere}.) We first observe 

\begin{lemma}\label{lem:empty}
    Let $\epsilon > 0$ and let $N_\epsilon \cong S \times [-\epsilon,\epsilon]$ be a small product neighborhood of $S$. If $\Gamma_S(T)=\emptyset$, then every loop in $T \cap \bd{N_\epsilon}$ is inessential in $T$. 
\end{lemma}

\begin{proof}
    We argue by contradiction. Suppose that $T \cap S \times \{-\epsilon,\epsilon\}$ contains a loop that is essential in $T$. Let $c$ be an innermost one among all such loops. Then, $c$ bounds a disk in $\bd{N_\epsilon}$ that is a shadow of some compressing disk for $T$. Thus, $\Gamma_S(T)$ is not empty.  
\end{proof}

We say that a subset $c$ of $T$ is \emph{negligible} if $c$ is contained in a disk in $T$. The item (1) in Proposition~\ref{prop:empty-contractible} implies that the cell structure of $T$ induced by $f$ is relatively simple, after removing negligible intersections by an isotopy:  

\begin{lemma}\label{lem:empty->cells}
    If $\Gamma_S(T)=\emptyset$, then there is an isotopy of $T$ that removes all the negligible intersections with $S$ while the other intersections are left invariant. After this isotopy, $T \setminus S$ is a collection of open $0$- and $2$-cells. 
\end{lemma}

\begin{proof}
    Choose $\epsilon>0$ such that there is no critical point of $f|_T$ in $N_\epsilon \setminus S$. By a standard innermost disk argument, we can convert $T$ into a surface $T'$ in $S \times (-\epsilon,\epsilon)$. See Figure~\ref{fig:pinch}. Since $S^3$ is irreducible, Lemma~\ref{lem:empty} implies that $T'$ is isotopic to $T$. By construction, there is no negligible intersection between $T'$ and $S$ while the other intersections are unchanged. Moreover, each component of $T' \setminus S$ contains exactly one critical point of $f|_{T'}$, which must be a center singularity. Thus, each component of $T' \setminus S$ must be either an open $0$- or $2$-cell, depending on whether it is above or below $S$. 
\end{proof}

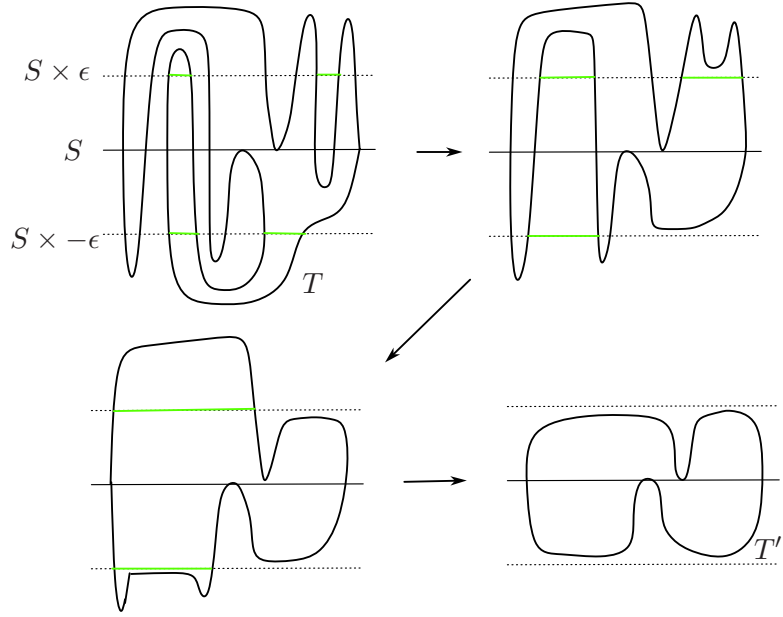
\begin{figure}
    \begin{overpic}[scale=.5]{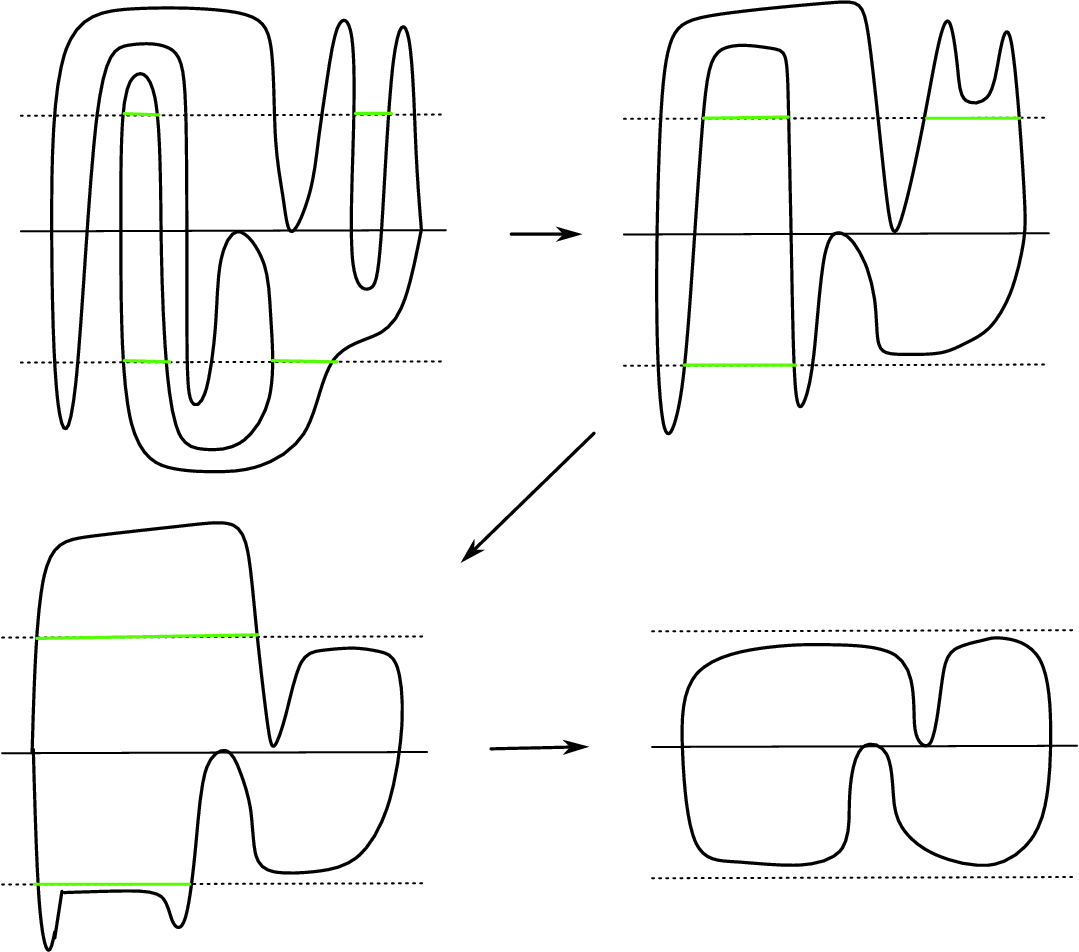}
        \put(80,120){$T$}
        \put(250,20){$T'$}
        \put(-10,170){$S$}
        \put(-25,200){$S \times \epsilon$}
        \put(-30,137){$S \times -\epsilon$}
    \end{overpic}    
    \caption{A schematic proof of Lemma~\ref{lem:empty->cells}. Green segments indicate innermost disks in $S \times \{\pm \epsilon\}$.} 
    \label{fig:pinch}
\end{figure}

\section{Sweepouts and the graphics}\label{sec:sweepout}

We here briefly review sweepout theory. We first see that $\Gamma(T)$ contains a homotopically nontrivial sphere, and then use it to construct a genus $g$ sweepout $\{T_t\}$. 

\subsection{A nontrivial sphere in the disk complex}\label{sec:nontrivial-sphere}

A key to our proof of Theorem~\ref{thm:Powell} is the following fact. 

\begin{lemma}[store=index]\label{lem:well-defined_index}
    The homotopy index of $\Gamma(T)$ is at most $2g-1$.
\end{lemma}

In fact, we can prove that the homotopy index of $\Gamma(T)$ is exactly $2g-1$. See Section~\ref{sec:bound4index} for details. 

\begin{proof}
    This follows directly from \cite[Theorem~3.2.2]{Campisi-Torres}. Nonetheless, we include a proof for completeness. 
    
    Let $T^*$ denote the surface $T$ with a marked point $* \in T$. The \emph{arc complex} $\mathcal{A}(T^*)$ is the simplicial complex whose $k$-simplexes are isotopy classes of systems of $k+1$ essential arcs $\{\alpha_i\}^k_{i=0}$ such that $\bd{\alpha_i}=*$ and $\interior \alpha_i \cap \interior \alpha_j=\emptyset$ for $i \neq j$. The \emph{arc complex at infinity} $\mathcal{A}_\infty(T^*)$ is the subcomplex of $\mathcal{A}(T^*)$ consisting of those simplexes that are represented by non-filling arc systems. In particular, $\mathcal{A}_\infty(T^*)$ contains the $(2g-2)$-skeleton of $\mathcal{A}(T^*)$. 

    Similarly, the \emph{curve complex} $\mathcal{C}(T)$ (resp. $\mathcal{C}(T^*)$) is the simplicial complex whose $k$-simplex is represented by a system of $k+1$ pairwise disjoint essential simple closed curves in $T$ (resp. $T^*$).  
    
    There are canonical homotopy equivalences 
    \begin{gather*}
        \Psi_0:\mathcal{A}^{\circ\circ}_\infty(T^*) \to \mathcal{C}^\circ(T^*),\\
        \Psi_1:\mathcal{C}(T^*) \to \mathcal{C}(T),  
    \end{gather*}
    where $\mathcal{A}^{\circ\circ}_\infty(T^*)$ and $\mathcal{C}^\circ(T^*)$ denote the second and first barycentric subdivision of $\mathcal{A}_\infty(T^*)$ and $\mathcal{C}(T^*)$, respectively. See, e.g.,  \cite{Broaddus} for a detailed description. (Roughly speaking, $\Psi_0$ maps each barycenter of a simplex spanned by essential arcs $\alpha_i$ to the barycenter of the simplex that is spanned by the boundary loops of a neighborhood of $\bigcup \alpha_i$ in $T$. On the other hand, $\Psi_1$ is given by forgetting the marked point.) 

    In fact, our lemma is a consequence of a result by Brouddus \cite{Broaddus}: Let $\sigma_0$ be a $(2g-1)$-simplex of $\mathcal{A}(T^*)$ shown in Figure~\ref{fig:arc_systtem}. Then, $\bd{\sigma_0} \subset \mathcal{A}_\infty(T^*)$ is a nontrivial sphere of dimension $2g-2$.  (In fact, $\bd{\sigma_0}$ represents a generator of $\lft{H}_{2g-2}(\mathcal{A}_\infty(T^*))$ as $\MCG(T^*)$-module.) Here, the point is that we can choose such $\sigma_0$ so that each arc lies entirely in the boundary of a compressing disk for $T$, as shown in Figure~\ref{fig:arc_systtem}. By winding the definitions of $\Psi_0$ and $\Psi_1$, we can see that the image of $\bd{\sigma_0}$ in $\mathcal{C}(T)$ is spanned by a collection of vertices, each of which is represented by a loop that bounds a compressing disk in $S^3$. Thus, this collection of disks must represent a nontrivial sphere in $\Gamma(T)$, which proves the lemma. 
\end{proof}

\begin{figure}
    \begin{overpic}[scale=.5]{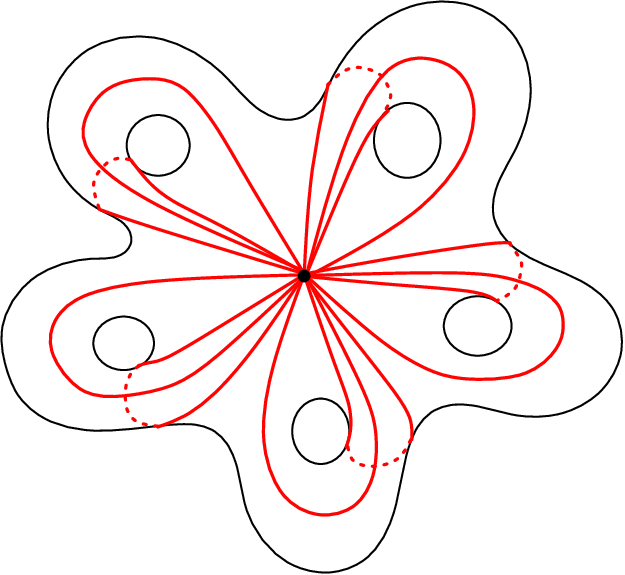}
    \end{overpic}    
    \caption{A $(2g-1)$-simplex $\sigma_0$ of $\mathcal{A}(T^*)$.}
    \label{fig:arc_systtem}
\end{figure}

\subsection{A sweepout by genus $g$ Heegaard surfaces}\label{sec:construction-sweepout}

In what follows, fix a triangulation $\sSph_0$ of $(d-1)$-sphere, $d \le 2g-1$, and a homotopically nontrivial simplicial map $\psi_0:\sSph_0 \to \Gamma(T)$. The existence of such a pair of $\sSph_0$ and $\psi_0$ is guaranteed by Lemma~\ref{lem:well-defined_index}. 

We will construct a sweepout $\{T_t\}_{t \in B^d}$ as follows. This can be done in three steps. First, for each vertex $v$ of $\sSph_0$, choose a compressing disk $D_v$ that represents $\psi_0(v)$. Let $C_v$ be a product neighborhood of $D_v$. Moreover, we may choose $C_v$ so that they satisfy the following property: If $v_0,v_1,\ldots,v_k$ span a $k$-simlex of $\sSph_0$, then $C_{v_i}$ are disjoint from each other. 

Next, we define $T_v$ for every vertex $v$ of $\sSph_0$. For $r \in [0,1]$, consider a map  $\rho_r:B^2 \times [-1,1]\to B^2 \times [-1,1]$ given by 
\begin{eqnarray*}
    \rho_r(x,u):=((1-r\beta(u))x,u),
\end{eqnarray*}
where $\beta:[0,1] \to [0,1]$ is a bump function. Fix an identification $\iota_v:B^2 \times [-1,1] \to C_v$ and put $\eta^r_v:=\iota_v \circ \rho_r \circ \iota^{-1}_v$. Define $T_v$ to be $\eta^1_v(T \cap C_v)$ in $C_v$ and  $T_v=T$ in the complement of $C_v$. 

Now viewing $B^d$ as a cone over $\sSph_0$, we have a triangulation $\sBall_0$ of $B^d$. Let $\sigma$ be a $d$-simplex spanned by $0=v_0,v_1,\ldots,v_d$. Identify $\sigma$ to a $d$-cube with its corners labeled by $v_0,v_1,\ldots,v_d$ as in Figure~\ref{fig:sweepout}. Every point $t \in \sigma$ can be written as $t=\sum r_k v_k$ for $r_k \in [0,1]$. Then, define $T_t$ by 
\begin{eqnarray*}
    T_t \cap \bigcup C_{v_i}=\eta^{r_1}_{v_1} \circ \eta^{r_2}_{v_2} \circ \cdots \circ \eta^{r_d}_{v_d}\left(T \cap \bigcup C_{v_i}\right)
\end{eqnarray*}
in $\bigcup C_v$ and $T_v=T$ in the complement of $\bigcup C_v$. Apply the same argument to other simplexes of $\sBall_0$ to extend $\{T_t\}$ over the entire space $B^d$. 

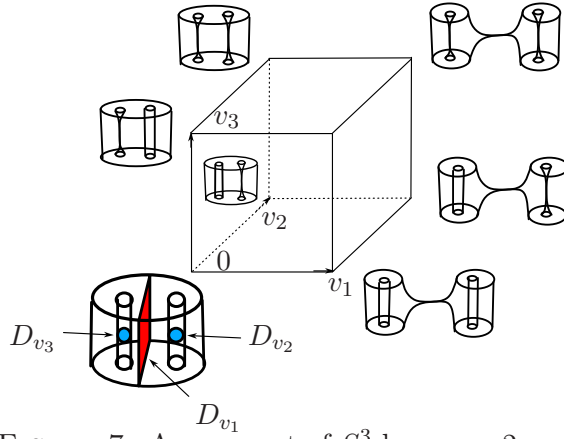
\begin{figure}
    \begin{overpic}[scale=.5]{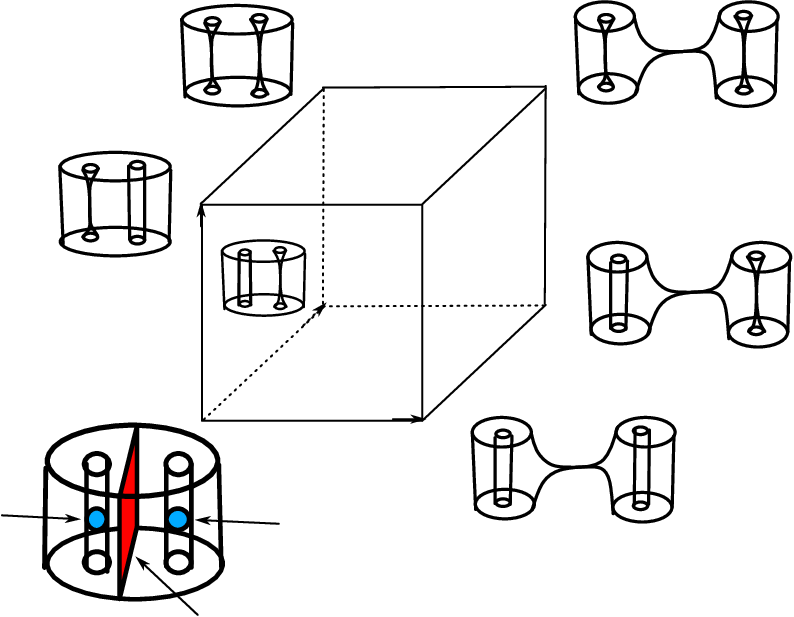}
        \put(50,-10){$D_{v_1}$}
        \put(70,20){$D_{v_2}$}
        \put(-20,20){$D_{v_3}$}
        \put(58,48){$0$}
        \put(100,40){$v_1$}
        \put(75,65){$v_2$}
        \put(57,103){$v_3$}
    \end{overpic}    
    \caption{A sweepout of $S^3$ by genus $2$ surfaces.}
    \label{fig:sweepout}
\end{figure}

We will refer to the intersection of $T_t$ with $C_v$ as a \emph{neck} of $T_t$. 

\subsection{The graphic}\label{sec:graphic} 

The \emph{graphic} determined by $\{T_t\}_{t \in B^d}$ and $\{S_s\}_{s \in [-1,1]}$ is the subset $\grph \subset B^d \times [-1,1]$ consisting of those points $(t,s)$ such that for some $u \in T_t \cap S_s$, $d(f|_{T_t})$ has rank $0$ at $u$. Here, we summarize basic properties of the graphic $\grph$. Although the parameter space for $t$ is $d$-dimensional, a local picture of the intersection between $T_t$ and $S_s$ is rather simple:  If $\{T_t\}$ and $\{S_s\}$ are in a general position, then the germ of $f|_{T_t}$ at $u \in T_t$ is either 
\begin{itemize}
    \item $f(x,y)=x$ ($u$ a regular point), 
    \item $f(x,y)=x^2+y^2$ ($u$ a center singularity), 
    \item $f(x,y)=x^2-y^2$ ($u$ a saddle singularity), or 
    \item $f(x,y)=x^3 \pm y^2$ ($u$ a birth/death singularity). 
\end{itemize}

The graphic $\grph$ consists of $d$-dimensional sheets, each of which corresponds to either a center or saddle singularity, as well as cuspidal loci that correspond to birth/death singularities. If $(t,s)$ lies in the intersection of $p$ sheets, then $T_t$ tangents $S_s$ in at most $p$ points. 

For our present purpose, saddle tangencies are only essential. So, we will focus on the \emph{reduced graphic} $\redg$, which is obtained from $\grph$ by deleting cuspidal loci and sheets that correspond to center singularities. For the convenience of references, we summarize these properties of $\redg$ in the following proposition. 

\begin{proposition}\label{prop:graphic}
    Suppose that $\{T_t\}$ and $\{S_s\}$ are in a general position. Then, $\redg$ is the union of $d$-dimensional sheets that intersect each other transversely. If $(t,s)$ lies in the intersection of $p$ sheets, then $T_t \cap S_s$ contains exactly $p$ saddle tangencies.   
\end{proposition}

\section{Lifting a manifold in the parameter space to the disk complex}\label{sec:balanced_sphere}

In this section, we see that arguments in \cite{Bachman} can be adapted to Heegaard spittings of $S^3$ to extract information about the disk complex from sweepouts. Let $\{T_t\}$ and $\{S_s\}$ be as before. The first proposition is a typical application of the double sweepout argument:  

\begin{proposition}\label{prop:balancedsphere}
    Let $g > 0$. Then, there exists a smooth function $z:B^d \to [-1,1]$ with the following property. For each $t$, $S_{z(t)}$ is a balancing sphere for $T_t$.  
\end{proposition}

\begin{proof}
    For $t \in B^d$, let $I_a(t)$ (resp. $I_b(t)$) be the subset of $[-1,1]$ consisting of those points $s$ such that $T_t$ is mostly above (resp. below) $S_s$. 
    Observe that 
    \begin{itemize}
        \item if $s$ is sufficiently close to $-1$, $T_t$ is mostly above $S_s$, and 
        \item if $-1<s'<s<1$ and $T_t$ is mostly above $S_s$, then $T_t$ is also mostly above $S_{s'}$. 
    \end{itemize}
    Thus, we conclude that there exists a point $s_a(t) \in (-1,1)$ such that $I_a(t)=[-1,s_a(t))$. A symmetric argument shows that there exists a point $s_b(t) \in (-1,1)$ such that $I_b(t)=(s_b(t),1]$. 
    
    We see that $s_a(t) \le s_b(t)$. Indeed, if $s_b(t)<s_a(t)$, then there is a point $s \in (s_a(t),s_b(t))$ such that $T_t$ is both mostly above and mostly below $S_s$. After perturbing $s$, we may assume that $T_t$ and $S_s$ intersect transversely. Then, any loop in $T_t \cap S_s$ bounds the two disks on opposite sides of each other. This shows that $T$ is diffeomorphic to $S^2$, contradicting the assumption. 
    
    Now define a function $z$ by choosing a point $z(t)$ in $[s_a(t),s_b(t)]$. 
\end{proof}

For brevity, we will write $S_t$ for $S_{z(t)}$ when there is no danger of confusion. By the definition of $T_t$, there is a canonical diffeomorphism $(S^3,T) \to  (S^3,T_t)$. We will identify $\Gamma(T)$ with $\Gamma(T_t)$ through this identification. For each $t$, let $\Sigma_t$ denote the preimage of $\Gamma_{S_t}(T_t) \subset \Gamma(T_t)$ in $\Gamma(T)$. 

In what follows, we will identify $B^d$ with the graph of the function $z:B^d \to [-1,1]$. The next proposition is essentially identical to Claim~3.3 of \cite{Bachman}. 

\begin{proposition}\label{prop:inclusions}
    Let $\sBall$ be a triangulation of $B^d$ such that the closure of $\redg \cap B^d$ is a subcomplex of $\sBall$. Let $\sigma$ be a simplex of $\sBall$. 
    \begin{enumerate}
        \item If $t,t' \in \interior \sigma$, then $\Sigma_t=\Sigma_{t'}$.
        \item If $t' \in \bd{\sigma}$ and $t \in \interior \sigma$, then $\Sigma_{t'} \subset \Sigma_t$. 
    \end{enumerate}
\end{proposition}

\begin{proof}
    Let $\mathcal{R}_p \subset B^d$ denote the set of points that lie in the intersection of exactly $p$ sheets of $\redg$. By definition, $B^d=\cup^{d+1}_{p=0} \mathcal{R}_p$.  

    First, suppose that $t,t' \in \interior \sigma$. Since the closure of $\redg \cap B^d$ is a subcomplex of $\sBall$, $t$ and $t'$ are in the same path component of $\mathcal{R}_p$. So, the intersection pattern of $T_t$ and $S_t$ is the same as that of $T_{t'}$ and $S_{t'}$, up to negligible intersections. This implies that $\Sigma_t=\Sigma_{t'}$, which proves (1). 

    Next, suppose that $t' \in \bd{\sigma}$ and $t \in \interior \sigma$. Then, $T_t$ and $S_t$ are obtained from $T_{t'}$ and $S_{t'}$ by resolving some saddle tangencies in $T_{t'} \cap S_{t'}$. This implies that $T_{t'} \setminus N \subset T_t \setminus N$ for some product neighborhood $N$ of $S_{t'}$. Thus, if a disk $D$ represents a vertex of $\Gamma_{S_{t'}}(T_{t'})$, $\bd{D}$ can be pushed off $N$ by an isotopy so that $D$ represents a vertex of $\Gamma_{S_t}(T_t)$. This proves (2). 
\end{proof}

The main result of this section is 

\begin{lemma}\label{lem:extension}
    Let $Y$ be a submanifold of $\interior B^d$. Suppose that $\Sigma_t \neq \emptyset$ for any $t \in Y$. Then, there exists a map $\psi:Y \to \Gamma(T)$ such that $\psi(t) \in \Sigma_t$ for any $t \in Y$. Moreover, $\psi$ is unique up to homotopy. 
\end{lemma}

Roughly speaking, Lemma~\ref{lem:extension} tells us that a submanifold in the parameter space can lift to $\Gamma(T)$ under an appropriate assumption, giving a converse direction of the correspondence from the disk complex to sweepouts described in Section~\ref{sec:sweepout}. The proof is based on ideas in \cite[Proof of Theorem~3.2]{Bachman}. The lemma can also be compared with \cite[Lemma~5.5]{Iguchi}, where a similar result is obtained using geometric ideas. 

\begin{proof}
    Let $\sBall$ be a triangulation of $B^d$ such that $\sBall$ contains $Y$ and the closure of $\grph \cap B^d$ as subcomplexes of $\sBall$. Let $Y^k$ denote the $k$-skeleton of $Y$. We construct a map $\psi$ by induction. For $t \in Y^0$, define $\psi(t)$ to be one of vertices of $\Sigma_t \neq \emptyset$. Now suppose that $\psi$ has already been defined on $Y^{k-1}$. Let $\sigma$ be a $k$-simplex of $Y$ and $t_0 \in \interior \sigma$. By induction and Proposition~\ref{prop:inclusions} (1), $\psi(t)$ is in $\Sigma_{t_0}$ for any $t \in \bd{\sigma}$. Since $\Sigma_{t_0}$ is contractible by Proposition~\ref{prop:empty-contractible}, $\psi$ extends over $\sigma$ so that $\psi(t) \in \Sigma_{t_0}$ for any $t \in \sigma$. This means that $\psi(t) \in \Sigma_t$ for any $t \in \interior \sigma$ as $\Sigma_{t_0}$ is independent of the choice of $t_0$ by Proposition~\ref{prop:inclusions} (2). Applying the same construction to each $k$-simplex defines  $\psi$ on $Y^k$, completing the induction. 
    
    Finally, if $\psi'$ is another map such that $\psi'(t) \in \Sigma_t$ for $t \in Y$, then there exists a homotopy between $\psi'$ and $\psi$ since $\Sigma_t$ is contractible by Proposition~\ref{prop:empty-contractible}. This shows that $\psi$ is unique up to homotopy. 
\end{proof}

\section{Digression: A lower bound on the homotopy index}\label{sec:bound4index}

In this section, we present the two applications (Theorems~\ref{thm:normal_position} and \ref{thm:h-index}) of the results obtained through Sections~\ref{sec:diskcomplex} to \ref{sec:balanced_sphere}. Although these results are not used in the proof of Theorem~\ref{thm:Powell}, their proofs can be seen as a prototype of the arguments in Section~\ref{sec:parallelization}. A key is the following lemma: 

\begin{lemma}\label{lem:local_incompressibility}
    Let $\{T_t\}$, $\{S_s\}$ and $z:B^d \to [-1,1]$ be as before. Then, there exists a point $t \in \interior B^d$ such that $\Sigma_t=\emptyset$. 
\end{lemma}

\begin{proof}
    Let $r \in (0,1)$ sufficiently close to $1$. We denote by $B_r \subset B^d$ the $d$-ball of radius $r$. As before, we identify $B^d$ with the graph of the function $z:B^d \to [-1,1]$. Let $\sBall$ be a triangulation of $B^d$ such that $\sBall$ contains $B_r$ and the closure of $\grph \cap B^d$ as subcomplexes. 
    
    Define a map from the vertices of $\bd{B_r}$ to $\Gamma(T)$ by sending each vertex $t$ to a compressing disk $D_t$ that corresponds to a thin neck of $T_t$. Note that $D_t$ can be taken to be disjoint from $S_t$ provided that $r$ is close to $1$. Thus, $D_t \in \Sigma_t$ and hence the above map extends to a map $\psi:\bd{B_r} \to \Gamma(T)$ by Lemma~\ref{lem:extension}. By construction, $\psi$ is homotopic to $\psi_0$. 

    If $\Sigma_t$ is not empty for all $t \in \interior B^d$, $\psi$ extends to a map $B_r \to \Gamma(T)$ by Lemma~\ref{lem:extension}, a contradiction since $\psi \simeq \psi_0$ is homotopically nontrivial. 
\end{proof}

The first consequence of Lemma~\ref{lem:local_incompressibility} tells us that $T$ can be isotoped into a ``normal form" with respect to $S$, which can be compared with Theorem~3.2 of \cite{Bachman}. 

\begin{theorem}[store=normal_position]\label{thm:normal_position}
    Let $T$ be a genus $g$ Heegaard surface of $S^3$ and $S$ a $2$-sphere in $S^3$. Then, $T$ can be isotoped so that 
    \begin{itemize}
        \item $T$ intersects $S$ transversely away from $p$ saddle tangencies, where $2g \le p \le d+1$, and 
        \item $\Gamma_S(T)=\emptyset$.  
    \end{itemize} 
\end{theorem}

\begin{proof}
    By Lemma~\ref{lem:local_incompressibility}, we may isotope $T$ so that $\Gamma_S(T)=\emptyset$. Let $p$ be the number of saddle tangencies in $T \cap S$. By Lemma~\ref{lem:empty->cells}, $T$ is isotopic to a surface $T'$ such that $T' \setminus S$ consists of open $0$- and $2$-cells. Let $p'$ be the number of saddle tangencies in $T' \cap S$. Then,  
    \begin{eqnarray*}
        2-2g=\chi(T) \ge 2-p' \ge 2-p. 
    \end{eqnarray*}
    Therefore, $2g \le p$. The inequality $p \le d+1 $ follows from Proposition~\ref{prop:graphic}. 
\end{proof}

In particular, it follows from Theorem~\ref{thm:normal_position} that $2g-1 \le d$. Putting together this and Lemma~\ref{lem:well-defined_index} reproduces 
a result by Appel \cite{Appel} and Campisi and Torres \cite{Campisi-Torres}, as mentioned in Section~\ref{sec:Introduction}. 

\begin{theorem}[{\cite{Appel, Campisi-Torres}}]\label{thm:h-index}
  Let $T$ be a genus $g$ Heegaard surface of $S^3$. Then, $\Gamma(T)$ has homotopy index $2g-1$.  
\end{theorem}

\section{Sliding a spine in a handlebody}\label{sec:sliding}

We now return to the proof of Theorem~\ref{thm:Powell}. Let $\{T_t\}$ and $\{S_t=S_{z(t)}\}$ be as before. In this section, we study an isotopy of the form $T_{l(\theta)}$, where $l:[0,1] \to B^d$ is a radial segment of $B^d$. The goal of this section is Theorem~\ref{thm:bouquet}. 

We say a properly embedded compact surface $P$ in a handlebody $A$ is a \emph{parallelism surface} if there exists a diffeomorphism from $(A,P)$ to $(P \times [-1,1],P \times \{0\})$. We start with the following well-known fact that is a consequence of the path connectedness of the arc complex $\mathcal{A}(P)$. 

\begin{proposition}\label{prop:arc-spine}
    Suppose that $P$ is a parallelism surface for a handlebody $A$. Then, any pair of spines $J,J' \subset P$ are related to each other by edge slides and isotopies inside $P$. 
\end{proposition}

\begin{proof}
    Since any spine can be slid slightly to be trivalent, we may assume that both $J$ and $J'$ are trivalent. Consider a simplicial complex $\mathcal{K}$ defined as follows. Every vertex of $\mathcal{K}$ is represented by a trivalent spine $J$ lying in $P$, and two such spines determine the same vertex if they are isotopic inside $P$. A $k$-simplex is represented by a collection of spines $\{J_i\}^k_{i=0}$ such that $J_i$ is related to $J_j$ by an edge slide inside $P$. It suffices to show that $\mathcal{K}$ is path connected. 
    
    Let $\{\alpha_i\}$ be an arc system that represents a simplex in $\mathcal{A}(P)$ of maximal dimension. Then, $\{\alpha_i\}$ cuts $P$ into a collection of hexagons and induces a cell decomposition of $P$. Observe that the $1$-skeleton of the dual of the decomposition is a spine of $A$. Conversely, any trivalent spine in $P$ is obtained in this way. Furthermore, we can see that two spines lying in $P$ are related by an edge slide inside $P$ if the corresponding simplexes in $\mathcal{A}(P)$ share a face. Since $\mathcal{A}(P)$ and its dual are path connected (in fact, contractible by \cite{Harer83}), so is $\mathcal{K}$. This proves the proposition. 
\end{proof}

We denote by $A_t$ the handlebody with $\bd{A_t}=T_t$ that lies in the same side as $A$. Let us consider a situation where $A_t \cap S_t$ is a parallelism surface in $A_t$ up to negligible intersections. More precisely, 

\begin{definition}
    We say $T_t$ is \emph{nearly parallel} to $S_t$ if the following hold. There exists an isotopy of $T_t$ that removes all the negligible intersections with $S_t$ while the other intersections are left invariant. After this isotopy, $P_t=A_t \cap S_t$ is a parallelism surface in $A_t$. 
\end{definition}

Note that by Lemma~\ref{lem:local_incompressibility} and the next proposition, such a situation actually occurs at some point in $B^d$. 

\begin{proposition}\label{prop:resolution}
    Suppose that $t$ is a point in $B^d$ and $\Sigma_t =\emptyset$. Then, there exists a point $\lft{t}$ near $t$ such that $T_{\lft{t}}$ is nearly parallel to $S_{\lft{t}}$. 
\end{proposition}

\begin{proof}
    By Lemma~\ref{lem:empty->cells}, $T_t$ is isotopic to a surface $T'_t$ such that $T'_t \setminus S_t$ is a collection of open $0$- and $2$-cells. Resolving saddle tangencies to increase $A'_t \cap S_t$ gives rise to a collection of surfaces that can be isotoped simultaneously into $S_t$. This means that $T'_t$ is nearly parallel to $S_t$. On the other hand, if we resolve the saddle tangencies in $T_t \cap S_t$ in the same manner to obtain a surface $T_{\lft{t}}$, then the intersection pattern between $T_{\lft{t}}$ and $S_{\lft{t}}$ is identical to that between $T'_t$ and $S_t$, up to negligible intersections. This shows that $T_{\lft{t}}$ is nearly parallel to $S_{\lft{t}}$. 
\end{proof}

Let $K \subset S$ as in Section~\ref{sec:outline}. The main result of this section is 

\begin{theorem}\label{thm:bouquet}
    Suppose that $l:[0,1] \to B^d$ is a segment with $l(0)=0$ and $l(1)=t_1$ such that $T_{t_1}$ is nearly parallel to $S_{t_1}$. Then, there exists a handlebody $V \subset S^3$ such that for each $\theta \in [0,1]$, $A_{l(\theta)}$ is contained in $V$ as a deformation retract. Moreover, for any spine $K'$ of $A_{l(1)}$ lying in $S_{t_1}$, there exist a family of spines $K_\theta \subset V$ such that 
    \begin{itemize}
        \item $K_0=K$ and $K_1=K'$, and 
        \item $K_\theta$ arises from a sequence of $S$- and $S_{t_1}$-slides as well as isotopies.  
    \end{itemize}
\end{theorem}

We define a handlebody $V$ as follows. Let $l:[0,1] \to B^d$ be as in Theorem~\ref{thm:bouquet} and let $\sigma$ be the smallest simplex of $\sBall_0$ that contains $l$. As before, view $\sigma$ as a $k$-cube in Euclidean space. Then, we can write $l$ in the form $l(\theta)=\sum_{i=1}^{k} \alpha_i(\theta)v_i$, each $v_i$ corresponding to a vertex of $\sigma$. Without loss of generality, we may assume that the first $m$ vertices are represented by disks in $B$ while the latter $k-m$ vertices are represented by disks in $A$. Define $V$ to be the handlebody corresponding to the point $x=\sum_{i=1}^m \alpha_i(1)v_i$ in $\sigma$ (Figure~\ref{fig:cube}). As $D_{v_1},\ldots,D_{v_m}$ lie in $B$, we have 
\begin{eqnarray}
    A_{l(\theta)} \subset V 
\end{eqnarray}
for all $\theta \in [0,1]$. This proves the first part of Theorem~\ref{thm:bouquet}. 

\begin{figure}
    \begin{overpic}[scale=.5]{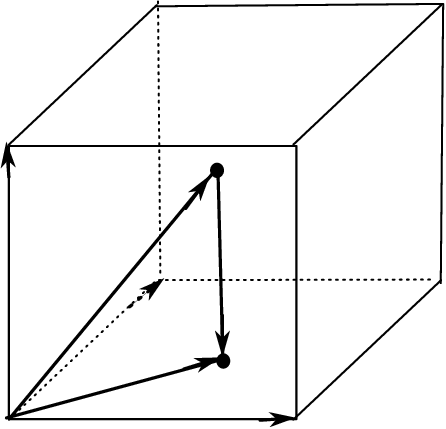}
        \put(-5,-5){$0$}
        \put(57,60){$t_1$}
        \put(50,7){$x$}
        \put(30,42){$l$}
        \put(25,12){$\ell$}
        \put(57,35){$\ell'$}
        \put(70,-3){$v_1$}
        \put(40,38){$v_2$}
        \put(-5,75){$v_3$}
    \end{overpic}
    \caption{Segments in $\sigma$.}
    \label{fig:cube}
\end{figure}
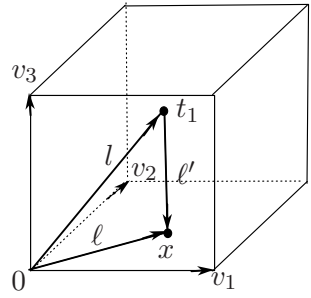

On the other hand, the second part of Theorem~\ref{thm:bouquet} is a consequence of Lemmas~\ref{lem:sliding} and \ref{lem:sliding2} below. To state these results, we need a few definitions. 

\begin{definition}
We say a spine $J$ is \emph{standard} if $J$ is a bouquet formed by loops $e^1,e^2,\ldots,e^g$ that bound disks $\Delta^1,\Delta^2,\ldots,\Delta^g$, respectively, in some level sphere such that $\interior \Delta^i \cap \interior \Delta^j=\emptyset$ for $i \neq j$.      
\end{definition}

\begin{definition}
    Suppose that $P$ is a subsurface of some level sphere. A standard spine $J \subset P$ is called \emph{peripheral} in $P$ if every loop of $J$ is boundary parallel in $P$. 
\end{definition}

A primary feature of a peripheral spine can be summarized as follows: 

\begin{lemma}\label{lem:b-parallel}
    Suppose that $V_{\bullet} \subset S^3$ is a genus $g$ handlebody that intersects a level sphere $S$ in a subsurface $P \subset S$. Suppose that $J \subset P$ is a standard peripheral spine of $V_{\bullet}$ formed by loops $e^1,e^2,\ldots,e^g$ that are parallel to circles $c^1,c^2,\ldots,c^g \subset \bd{P}$, respectively. Then, the circles $c^j$ bound pairwise disjoint disks in the complement of $\interior V_{\bullet}$. 
\end{lemma}

\begin{proof}
    For each $j$, choose a disk $\Delta^j \subset S$ such that $\bd{\Delta^j}=c^j$ and $\Delta^j \cap \Delta^i=\emptyset$ for $j \neq i$. If $\Delta^j \cap \interior V_{\bullet}=\emptyset$ for all $j$, there is nothing to prove. If this is not the case, we can isotope $\Delta^j$ to be disjoint from $\interior V_{\bullet}$ since $\interior \Delta^j \cap V_{\bullet}$ lies in $V_{\bullet} \setminus J \cong \bd{V_{\bullet}} \times (0,1]$. 
\end{proof}

The first step of the proof of Theorem~\ref{thm:bouquet} is 

\begin{lemma}\label{lem:sliding}
    Put $P:=V \cap S$. Then, there exists a standard peripheral spine $J \subset P$ such that $J$ is related to $K$ by a sequence of $S$-slides and isotopies inside $V$. 
\end{lemma}

\begin{proof}
    Consider a segment $\ell:[0,1] \to \sigma$ with $\ell(0)=0$ and $\ell(1)=x$. See Figure~\ref{fig:cube}. Put $V_r:=A_{\ell(r)}$. By definition $V_0=A$, $V_1=V$ and $V_r \subset V_{r'}$ for $0 \le r \le r' \le 1$. 

    Let $0<r_1<r_2<\cdots<r_{n-1}<r_n<1$ be the values of $r$ where $\bd{V_r}$ tangents to $S$ in a saddle point. Put $P_r:=V_r \cap S$. By definition, $P_0=A \cap S$ and $P_1=P$. It is sufficient to find a finite sequence of standard spines $J_i \subset S$, $i=1,2,\cdots,n$, with the following properties. 
    \begin{itemize}
        \item $J_0=K$. 
        \item $J_i$ is related to $J_{i-1}$ by a sequence of $S$-slides and isotopies inside $V$. 
        \item For $r \in (r_i,r_{i+1})$, $J_i \subset P_r$ and $J_i$ is peripheral in $P_r$.  
    \end{itemize}
    Indeed, once such $J_i$ are found, we can define $J=J_n$.  

    Set $J_0=K$. It is evident that $J_0$ is peripheral in $P_0$. We argue by induction. Suppose that $J_{i-1}$ has already been defined. Since $J_{i-1}$ is standard, $J_{i-1}$ is formed by loops $e^j_{i-1} \subset S$, $j=1,2,\ldots,g$. Moreover, as $J_{i-1}$ is peripheral in $P_r$ for $r \in (r_{i-2},r_{i-1})$, there is a circle $c^j_{i-1} \subset \bd{P_r}$  such that $e^j_{i-1}$ is parallel to $c^j_{i-1}$ in $P_r$. 

    \begin{claim*}
        When we pass from $r_i-\epsilon$ to $r_i+\epsilon$, all the loop $e^j_{i-1}$ are still parallel to $\bd{P_{r_i+\epsilon}}$ but at most one loop. 
    \end{claim*}

    \begin{proof}
        Note that $P_{r_i+\epsilon}$ is obtained from $P_{r_i-\epsilon}$ by attaching a 2D $1$-handle. Since $P_r$ is a subsurface of a $2$-sphere, the $1$-handle meets at most one of $c^j_{i-1}$. 
    \end{proof}

    Without loss of generality, we may assume that $e^1_{i-1}$ is the only loop that is not boundary parallel in $P_{r_i+\epsilon}$. By Lemma~\ref{lem:b-parallel}, the circles $c^j_{i-1}$, $j \neq 1$, bound pairwise disjoint disks in the complement of $\interior V_{r_i+\epsilon}$. Attaching $2$-handles to $V_{r_{i+\epsilon}}$ along these disks gives rise to an unknotted solid torus $U$. Thus, $e^1_{i-1}$ is isotopic to some loop $c \subset \bd{P_{r_i+\epsilon}}$ in $U$. This implies that $e^1_{i-1}$ is isotopic to $c$ in $V_{r_i+\epsilon}$ after a sequence of $S$-slides. See Figure~\ref{fig:sliding}. 
    
    Now define $J_i \subset S$ to be a standard spine formed by the loops 
    \begin{eqnarray*}
        c, e^2_{i-1},\ldots,e^g_{i-1}.
    \end{eqnarray*} 
    By construction, $J_i$ is related to $J_{i-1}$ by a sequence of $S$-slides followed by an isotopy within $V_{r_i+\epsilon}$. Moreover, $J_i$ is peripheral in $P_{r_i+\epsilon}$. This completes the inductive step and hence the proof of the lemma.  
\end{proof}

\begin{figure}
        \begin{overpic}[scale=.6]{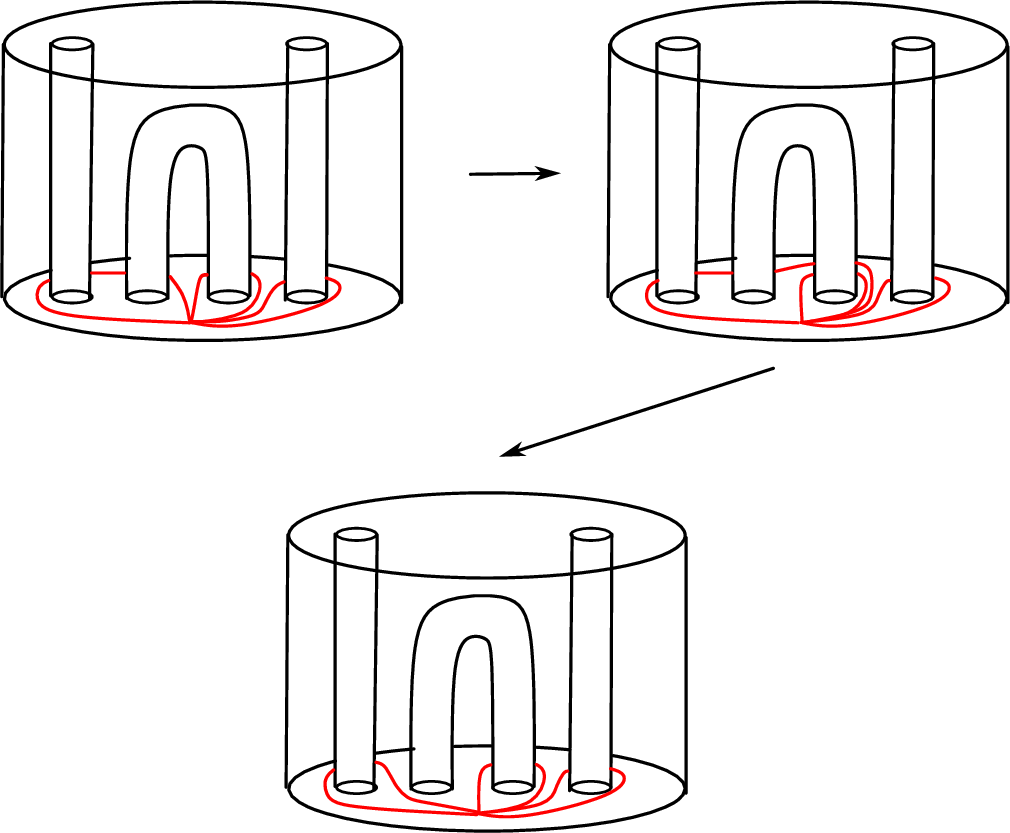}
            \put(50,225){$V_{r_i+\epsilon}$}
            \put(100,152){$J_{i-1}$}
            \put(180,10){$J_i$}
        \end{overpic}
        \caption{$J_i$ is related to $J_{i-1}$ by a sequence of $S$-slides followed by an isotopy.}
        \label{fig:sliding}
\end{figure}

The next step of the proof of Theorem~\ref{thm:bouquet} is the following lemma, which is essentially identical to Lemma~\ref{lem:sliding}. 

\begin{lemma}\label{lem:sliding2}
    Let $Q:=V \cap S_{t_1}$. Then, there exists a standard peripheral spine $J' \subset Q$ such that $J'$ is related to $K'$ by a sequence of $S_{t_1}$-slides and isotopies inside $V$. 
\end{lemma}

\begin{proof}
    Let $\ell':[0,1] \to \sigma$ be a segment with $\ell'(0)=t_1$ and $\ell'(1)=x$. See Figure~\ref{fig:cube}.  Put $V'_r:=A_{\ell'(r)}$. By definition $V'_0=A_{t_1}$ $V'_1=V$ and $V'_r \subset V'_{r'}$ for $0 \le r \le r' \le 1$. 
    
    As before, put $Q_r:=V'_r \cap S_{t_1}$. By definition, $Q_0=A_{t_1} \cap S_{t_1}$ and $Q_1=Q$. Since $T_{t_1}$ is nearly parallel to $S_{t_1}$, there is a standard peripheral spine $K'' \subset Q_0$. By replacing $K'$ with $K''$ via Proposition~\ref{prop:arc-spine}, we may assume that $K'$ is standard and peripheral in $Q_0$. Now apply the same argument as the proof of Lemma~\ref{lem:sliding} to obtain a sequence of standard peripheral spines $J'_i \subset Q$, $i=1,2,\ldots,n$. Then, $J'=J'_n$ satisfies the desired properties. 
\end{proof}

We now complete the proof of Theorem~\ref{thm:bouquet}. 

\begin{proof}[Proof of Theorem~\ref{thm:bouquet}]
    Let $J$ and $J'$ be as given in Lemmas~\ref{lem:sliding} and \ref{lem:sliding2}. Then, the loops of $J$ are isotopic to circles $c^1,c^2,\ldots,c^g \subset \bd{V}$. By Lemma~\ref{lem:b-parallel}, $c^i$ bound parwise disjoint compressing disks in $W=S^3 \setminus \interior V$, and these disks cut $W$ into a $3$-ball. Similarly, the loops of $J'$ are isotopic to circles $d^1,d^2,\ldots,d^g \subset \bd{V}$ with the same property. Since $c^j$ and $d^i$ are disjoint from each other, $\bigcup c^j$ must be isotopic to $\bigcup d^j$ in $\bd{V}$. Thus, $J$ and $J'$ are isotopic in $V$ (Figure~\ref{fig:flipping}). This together with Lemmas~\ref{lem:sliding} and \ref{lem:sliding2} implies that $K'$ is related to $K$ by a sequence of $S$-, $S_{t_1}$-slides and isotopies inside $V$.  
\end{proof}

\begin{figure}
        \begin{overpic}[scale=.6]{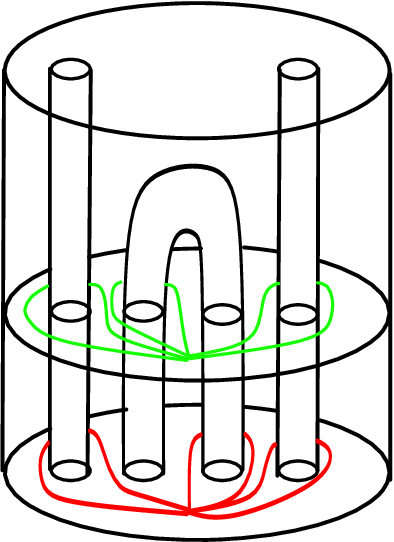}
            \put(-15,15){$S$}
            \put(-15,60){$S_{t_1}$}
        \end{overpic}
        \caption{$J$ (red) is isotopic to $J'$ (green) in $V$.}
        \label{fig:flipping}
\end{figure}

We conclude this section by rephrasing Theorem~\ref{thm:bouquet} in a form that is actually used in the next section. 

\begin{corollary}\label{cor:bouquet}
    Suppose that $l:[0,1] \to B^d$ is a segment with $l(0)=0$ and $l(1)=t_1$ such that $T_{t_1}$ is nearly parallel to $S_{t_1}$. Let $K' \subset S_{t_1}$ be a spine of $A_{t_1}$. Then, the isotopy $T_{l(\theta)}$ is equivalent to $T'_\theta$ such that $T'_\theta$ is supported by a family of spines $K_\theta$ with $K_0=K$ and $K_1=K'$ that arises from a sequence of $S$- and $S_{t_1}$-slides as well as isotopies. 
\end{corollary}

\begin{proof}
    Let $V$ and $K_\theta$ be as given in Theorem~\ref{thm:bouquet}. After enlarging $V$ slightly, we may assume that $A_{l(\theta)} \subset \interior V$. Define an isotopy of a Heegaard surface by $T'_\theta=\bd{N(K_\theta)}$, where $N(K_\theta)$ is a regular neighborhood of $K_\theta$. We can view both $T_{l(\theta)}$ and $T'_\theta$ as the boundaries of product neighborhoods of $\bd{V}$. Since the space of product neighborhoods is contractible, $T_{l(\theta)}$ is equivalent to $T'_\theta$.  
\end{proof}

\section{Parallelizing Heegaard surfaces}\label{sec:parallelization}

In this section, we establish a key step for the proof of Theorem~\ref{thm:Powell}. The argument here is basically the same as the proof of Theorem~\ref{thm:normal_position}, but one more parameter is introduced. 

Let $\{T_t\}$ be as before. Let $\tau_\theta:S^3 \to S^3$, $\theta \in [0,2\pi]$, be an isotopy with $\tau_0=\id_{S^3}$ and $\tau_{2\pi}(T)=T$. Then, define a new sweepout $\{T_{t\theta}\}$ by 
\begin{eqnarray*}
    T_{t\theta}:=\tau_\theta(T_t). 
\end{eqnarray*}

We can easily see that the arguments in Sections~\ref{sec:sweepout} and \ref{sec:balanced_sphere} are still valid in this setting. We can define the graphics $\grph \subset B^d \times [0,2\pi] \times [-1,1]$ as in Section~\ref{sec:sweepout}, which consists of $(d+1)$-dimensional sheets that correspond to center or saddle singularities, as well as $d$- and $(d-1)$-dimensional loci that correspond to singularities of higher codimension. As before, we focus on the reduced graphic $\redg$, each of whose sheets corresponds to a saddle tangency. The next two propositions follow by the same arguments as Propositions~\ref{prop:graphic} and \ref{prop:balancedsphere}, respectively: 

\begin{proposition}\label{prop:balancedsphere2}
    There exists a function $z:B^d \times [0,2\pi]\to [-1,1]$ with the following property. For each $(t,\theta)$, $S_{z(t,\theta)}$ is a balanced sphere for $T_{t\theta}$. 
\end{proposition}

\begin{proposition}\label{prop:graphic2}
    Suppose that $\{T_{t\theta}\}$ and $\{S_s\}$ are in general position. Then, $\redg$ is the union of $(d+1)$-dimensional sheets that intersect each other transversely. If $(t,\theta,s)$ lies in the intersection of $p$ sheets, then $T_{t\theta} \cap S_s$ contains exactly $p$ saddle tangencies.  
\end{proposition}

We write $S_{t\theta}=S_{z(t,\theta)}$ for short. Denote by $\Sigma_{t\theta}$ the preimage of $\Gamma_{S_{t\theta}}(T_{t\theta})$ in $\Gamma(T)$ under the canonical identification $(S^3,T) \to (S^3,T_{t\theta})$. 

We will identify $B^d \times [0,2\pi]$ with the graph of the function $z:B^d \times [0,2\pi] \to [-1,1]$, and fix a triangulation $\sCyl$ of $B^d \times [0,2\pi]$ such that $\sCyl$ contains the closure of $\redg \cap B^d \times [0,2\pi]$ as a subcomplex. 

\begin{proposition}\label{prop:inclusions2}
    Let $\sigma$ be a simplex of $\sCyl$. 
    \begin{enumerate}
        \item If $(t,\theta), (t',\theta') \in \interior \sigma$, then $\Sigma_{t\theta}=\Sigma_{t'\theta'}$. 
        \item If $(t',\theta') \in \bd{\sigma}$ and $(t,\theta) \in \interior \sigma$, then $\Sigma_{t'\theta'} \subset \Sigma_{t\theta}$. 
    \end{enumerate}
\end{proposition}

\begin{proof}
    This is the same as the proof of Proposition~\ref{prop:inclusions}. 
\end{proof}

\begin{lemma}\label{lem:extension2}
    Let $Y$ be a submanifold of $\interior B^d \times [0,2\pi]$. Suppose that $\Sigma_{t\theta} \neq \emptyset$ for any $(t,\theta) \in Y$. Then, there exists a map $\psi:Y \to \Gamma(T)$ such that $\psi(t,\theta) \in \Sigma_{t\theta}$ for any $(t,\theta) \in Y$. Moreover, $\psi$ is unique up to homotopy. 
\end{lemma}

\begin{proof}
    This follows by the same argument as the proof of Lemma~\ref{lem:extension} using Propositions~\ref{prop:empty-contractible} and \ref{prop:inclusions2}. 
\end{proof}

Next, we generalize the argument in Section~\ref{sec:bound4index} to the present setting. Consider the subset $L \subset \interior B^d \times [0,2\pi]$ given by  
\begin{equation*}
    L:=\{(t,\theta) \in \interior B^d \times [0,2\pi] \mid \Sigma_{t\theta} = \emptyset\}.
\end{equation*}
To prove the key lemma of this section, we need the following two facts. 
 
\begin{lemma}
    The subset $L$ is contained in the $1$-skeleton of $\sCyl$. 
\end{lemma}

\begin{proof}
    Let $(t,\theta) \in L$. It suffices to show that $(t,\theta)$ lies in the intersection of $d+1$ sheets of $\redg$. Suppose that $(t,\theta)$ lies in the intersection of $p \ge 0$ sheets. By Proposition~\ref{prop:graphic2}, $T_{t\theta} \cap S_{t\theta}$ contains $p$ saddle tangencies. By Lemma~\ref{lem:empty->cells}, $T_{t\theta}$ is isotopic to a surface $T'_{t\theta}$ such that $T'_{t\theta} \setminus S_{t\theta}$ is a collection of open $0$- and $2$-cells. Since $T'_{t\theta}$ tangents $S_{t\theta}$ in $p' \le p$ saddle points, we have 
    \begin{eqnarray*}
        2-2g=\chi(T_{t\theta}) \ge 2-p' \ge 2-p. 
    \end{eqnarray*}
    On the other hand, $d \le 2g-1$ by Proposition~\ref{lem:well-defined_index}. Comparing these two inequalities, we conclude that $d+1 \le p$. 
\end{proof}

\begin{proposition}\label{prop:sweepout_near_boundary2}
    Let $B_r \subset B^d$ be the $d$-ball of radius $r>0$. If $r$ sufficiently close to $1$, $L$ is contained in $B_r \times [0,2\pi]$. 
\end{proposition}

\begin{proof}
    Suppose that $r$ is sufficiently close to $1$. If $(t,\theta)$ is a point in the complement of $B_r \times [0,2\pi]$, then a compressing disk that corresponds to a thin neck of $T_{t\theta}$ does not intersect $S_{t\theta}$. In particular, $\Sigma_{t\theta} \neq \emptyset$ and $(t,\theta) \not \in L$. Therefore, $L$ must be contained in $B_r \times [0,2\pi]$. 
\end{proof}

Now we can prove the key lemma, which generalizes Lemma~\ref{lem:local_incompressibility}: 

\begin{lemma}\label{lem:local_incompressibility2}
    There exists an arc $\delta:[0,1] \to L$ that connects $B^d \times \{0\}$ to $B^d \times \{2\pi\}$.  
\end{lemma}

\begin{proof}
    Let $r>0$ be sufficiently close to $1$. By passing to subdivisions, we may assume that $B_r \subset B^d \times \{0\}$ is a subcomplex of $\sCyl$. We define a map from the vertices of $\bd{B_r}$ to $\Gamma(T)$ by sending each vertex $t$ to a compressing disk $D_t$ that corresponds to a thin neck of $T_t$. Since $D_t \in \Sigma_t$, this map extends to $\psi:\bd{B_r} \to \Gamma(T)$ by Lemma~\ref{lem:extension}. Moreover, $\psi$ is homotopic to $\psi_0$. 

    Now suppose, contrary to our claim, that there does not exist an arc $\delta:[0,1] \to L$ that connects $B^d \times \{0\}$ to $B^d \times \{2\pi\}$. Let $L'$ denote the union of those path components of $L$ that intersect $B^d \times \{0\}$. By assumption and Proposition~\ref{prop:sweepout_near_boundary2}, $L'$ intersects the boundary of $B^d \times [0,2\pi]$ in finitely many points in $B_r$. Let $Y$ be the boundary of a small regular neighborhood of $B_r \cup L'$ (Figure~\ref{fig:regular_nbd}). Then, $Y$ is a compact orientable $d$-manifold with $\bd{Y}=\bd{B_r}$. Since $Y$ does not meet $L$, $\Sigma_{t\theta}$ is not empty for any $(t,\theta) \in Y$. Thus, we can extend $\psi$ over $Y$ via Lemma~\ref{lem:extension2}. But this is a contradiction since $\psi \simeq \psi_0$ is not homologically trivial by Hurewicz's theorem.  
\end{proof}

\begin{figure}
        \begin{overpic}[scale=.5]{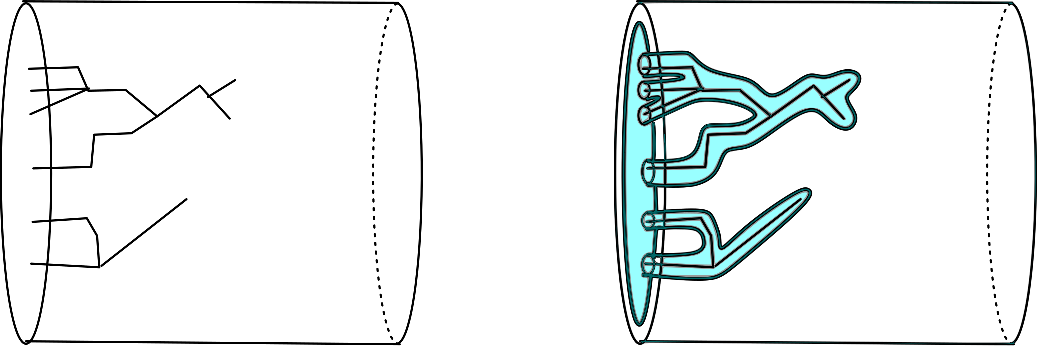}
            \put(52,42){$L'$}
            \put(200,42){$Y$}
        \end{overpic}
        \caption{A small neighborhood of $L' \cup B_r$.}
        \label{fig:regular_nbd}
\end{figure}

As a consequence of Lemma~\ref{lem:local_incompressibility2} we have 

\begin{theorem}[store=parallelism]\label{thm:parallelization}
    There exists an arc $\gamma:[0,2\pi] \to B^d \times [0,2\pi]$ with the following properties: 
    \begin{itemize}
        \item $\gamma|_{[0,2\pi/3]}$ is a radial segment in $B^d \times \{0\}$, while $\gamma|_{[4\pi/3,2\pi]}$ is a radial segment in $B^d \times \{2\pi\}$. 
        \item For every $\theta \in [2\pi/3,4\pi/3]$, $T_{\gamma(\theta)}$ is nearly parallel to $S_{\gamma(\theta)}$. 
    \end{itemize}
    See Figure~\ref{fig:arc}.
\end{theorem}

\begin{proof}
    Let $\delta:[0,1] \to B^d \times [0,2\pi]$ be an arc given in Lemma~\ref{lem:local_incompressibility2}. Then, $\Sigma_{\delta(\theta)}$ is an empty set for every $\theta \in [0,1]$. Applying Proposition~\ref{prop:resolution} for each $\theta$ gives rise to an arc $\lft{\delta}:[0,1] \to B^d \times [0,2\pi]$ such that $T_{\lft{\delta}(\theta)}$ is nearly parallel to $S_{\lft{\delta}(\theta)}$. Let  $l_0:[0,1] \to B^d \times \{0\}$ and $l_{2\pi}:[0,1] \to B^d \times \{2\pi\}$ denote the radial segments with $l_0(1)=\lft{\delta}(0)$ and $l_{2\pi}(1)=\lft{\delta}(1)$. Now we can define $\gamma$ to be the concatenation of $l_0$, $\lft{\delta}$ and $l_{2\pi}$. 
\end{proof}

\begin{figure}
        \begin{overpic}[scale=1]{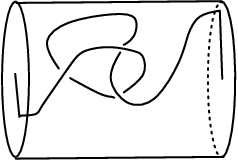}
        \end{overpic}
        \caption{The arc $\gamma$ in $B^d \times [0,2\pi]$.}
        \label{fig:arc}
\end{figure}

Finally, we unify the results in Sections~\ref{sec:sliding} and \ref{sec:parallelization} into a single theorem, which can be compared with Corollary~3 of \cite{Powell}: 

\begin{theorem}[store=planar]\label{thm:planargraph}
    If $\tau_\theta:S^3 \to S^3$, $\theta \in [0,2\pi]$, is an ambient isotopy representing an element of $G(S^3,T)$, then $\tau_\theta$ can be modified so that $T_\theta:=\tau_\theta(T)$ is supported by a family of spines $K_\theta$ with $K_0=K_{2\pi}=K$ that arises from a sequence of $S$-slides and isotopies. 
\end{theorem}

\begin{proof}
    Let $\gamma:[0,2\pi] \to B^d \times [0,2\pi]$ be an arc given in Theorem~\ref{thm:parallelization}. Since $\gamma$ is homotopic to the segment $\{0\} \times [0,2\pi]$, $T_{\gamma(\theta)}$ is equivalent to $T_\theta$. By replacing $T_\theta$ with $T_{\gamma(\theta)}$, we may assume that $T_\theta$ is nearly parallel to $S_\theta$ for $\theta \in [2\pi/3,4\pi/3]$. In particular, $T_\theta$ is supported by a family of spines $K_\theta \subset S_\theta$ for $\theta \in [2\pi/3,4\pi/3]$. On the other hand, we can use Corollary~\ref{cor:bouquet} to find spines $K_\theta$ for $\theta \in [0,2\pi/3]$ such that $T_\theta=\bd{N(K_\theta)}$ and $K_\theta$ arises from a sequence of $S$- and $S_{2\pi/3}$-slides as well as isotopies. A symmetric argument shows that there also exist spines $K_\theta$ for $\theta \in [4\pi/3,2\pi]$ with the same properties. 

    Now apply for each $\theta \in [0,2\pi]$ the ambient isotopy that takes $S_\theta$ to $S=S_0$. This gives an isotopy $T'_\theta$ and a family of spines $K'_\theta$ satisfying the desired properties. 
\end{proof}

\section{The proof of the Powell conjecture}\label{sec:finalstep}

In this section, we complete the proof of Theorem~\ref{thm:Powell}. We will prove the inverse direction of the argument in \cite{Powell}: We deduce the Powell conjecture from Theorem~\ref{thm:planargraph}, which is slightly weaker than Corollary~3 of \cite{Powell}. 

We say a compressing disk $a \subset A$ is \emph{primitive} if there exists a compressing disk $b \subset B$ (called a \emph{dual disk} for $a$) such that $\bd{b}$ intersects $\bd{a}$ in a single point. For $j=1,2,\ldots,g$, let $a^j$ be a cocore of the $j$th handle of $A$, and $b^j$ a cocore of the $1$-handle of $B$ that corresponds to $\Delta^j$. A consequence of Theorem~\ref{thm:planargraph} is the following:   

\begin{lemma}\label{lem:primitive_disks}
    If $\tau_\theta:S^3 \to S^3$, $\theta \in [0,2\pi]$, is an ambient isotopy representing an element of $G(S^3,T)$, then $\tau_\theta$ can be modified to satisfy the following property: There exists a finite sequence of primitive disks $a_{\theta_i} \subset A_{\theta_i}$ with dual disks $b_{\theta_i}$ such that 
    \begin{enumerate}
        \item $a_0=a_{2\pi}=a^1$,  
        \item $\tau^{-1}_{\theta_i}(a_{\theta_i})$ and $\tau^{-1}_{\theta_{i+1}}(a_{\theta_{i+1}})$ are disjoint, and 
        \item $\tau^{-1}_{\theta_i}(b_{\theta_i})$ intersects $\tau^{-1}_{\theta_{i+1}}(a_{\theta_{i+1}})$ in at most one point. 
    \end{enumerate}
\end{lemma}

\begin{proof}
    By Theorem~\ref{thm:planargraph}, we can modify $T_\theta$ so that the following holds: There are finitely many points $0=\theta_0<\theta_1<\cdots<\theta_{n-2}<\theta_{n-1}=2\pi$ and a sequence of spines $K_{\theta_i} \subset S$ such that $K_{\theta_i}$ related to $K_{\theta_{i-1}}$ by either an $S$-slide or an isotopy. 
    
    First, we define $a_{\theta_i}$ for $0<i<n$ by induction. Assume that $a_{\theta_{i-1}}$ has already been defined. If $K_{\theta_i}$ is related to $K_{\theta_{i-1}}$ by an isotopy, we can define $a_{\theta_i}$ as the image of $a_{\theta_{i-1}}$. If $K_{\theta_i}$ is related to $K_{\theta_{i-1}}$ by an $S$-slide about an edge $e_{i-1} \subset K_{\theta_{i-1}}$, then define $a_{\theta_i}$ to be a cocore of the $1$-handle whose core is the image of $e_{i-1}$. 

    We see that (2) holds. Observe that for $0 \le i < n$ the preimage of $a_{\theta_{i+1}}$ is disjoint from $a_{\theta_i}$. See Figure~\ref{fig:primitive_disks}. Furthermore, $a_{\theta_{n-1}}$ is disjoint from $a^1=a_{\theta_n}$ since $K_{\theta_{n-1}}=K$ and $a_{\theta_{n-1}}=a^j$ for some $1 \le j \le g$. Thus, (2) holds for any $0 \le i \le n$. 
    
    For each $i$ choose a compressing disk $b_{\theta_i} \subset S$ that intersects $a_{\theta_i}$ in a single point. (There are exactly two possible choices for such $b_{\theta_i}$.) Note the preimage of $\bd{a_{\theta_{i+1}}}$ in $T_{\theta_i}$ intersects $\bd{b_{\theta_i}}$ in at most one point (Figure~\ref{fig:primitive_disks}). Thus, the item (3) holds. 
\end{proof}

\begin{figure}
    \begin{overpic}[scale=.5]{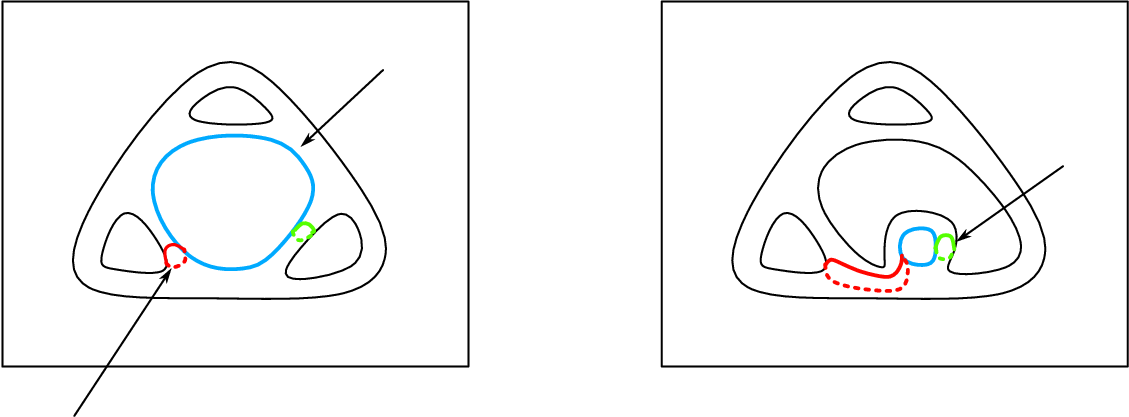}
    \end{overpic}
    \put(-220,-5){$T_{\theta_i}$}
    \put(-60,-5){$T_{\theta_{i+1}}$}
    \put(-265,80){$S$}
    \put(-262,-7){$\bd{a_{\theta_i}}$}
    \put(-182,84){$\bd{b_{\theta_i}}$}
    \put(-20,65){$\bd{a_{\theta_{i+1}}}$}
    \caption{Left: in $T_{\theta_i}$, the preimage of $a_{\theta_{i+1}}$ (green) is disjoint from $a_{\theta_i}$ (red) while it may intersect $b_{\theta_i}$ (blue) in a single point. Right: $a_{\theta_{i+1}}$ and the images of $a_{\theta_i}$ and $b_{\theta_i}$ in $T_{\theta_{i+1}}$.}
    \label{fig:primitive_disks}
\end{figure}
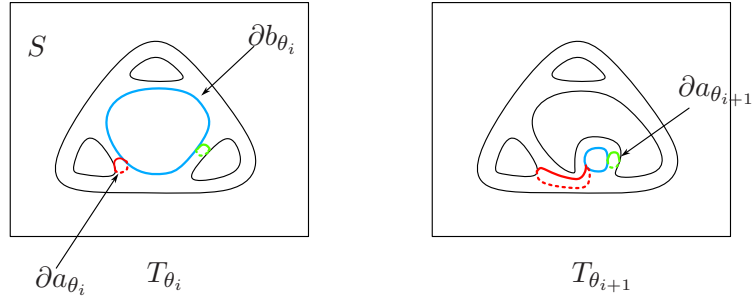

We now complete the proof of Theorem~\ref{thm:Powell}. We argue by induction on $g$. If $g=3$, the conclusion follows from \cite{Freedman-Scharlemann} or \cite{Cho-et-al}. Assume that $g > 3$ and the assertion holds for a genus $g-1$ Heegaard splitting. 

\begin{definition}[{\cite{Freedman-Scharlemann}}]
    Suppose that $T'$ is a Heegaard surface isotopic to $T$. We say two maps $h_1,h_2:(S^3,T') \to (S^3,T)$ are \emph{Powell equivalent} if $h_2h_1^{-1}$ is a Powell move. 
\end{definition}

\begin{lemma}\label{lem:P-equivalence}
    The choice of a primitive disk $a'$ for $T'$ uniquely determines a map 
    \begin{eqnarray*}
       h:(S^3,T',a') \to (S^3,T,a^1) 
    \end{eqnarray*}
    up to Powell equivalence. 
\end{lemma}

\begin{proof}
    This follows from induction and \cite[Corollary~3.6]{Freedman-Scharlemann}. 
\end{proof}

Let $\tau_\theta:S^3 \to S^3$ be an isotopy that represents an element of $G(S^3,T)$. It suffices to show that $\tau_{2\pi}$ is Powell equivalent to $\id_{S^3}$. Let $a_{\theta_i}$ be a sequence of disks given in Lemma~\ref{lem:primitive_disks}. For each $i$, fix a map 
\begin{eqnarray*}
    h_{\theta_i}:(S^3,T_\theta,a_{\theta_i},b_{\theta_i}) \to (S^3,T,a^1,b^1).
\end{eqnarray*}
We may take $h_0$ and $h_{2\pi}$ to be $\id_{S^3}$. Put $a'_i:=\tau_{\theta_i}^{-1}(a_{\theta_i})$ and $b'_i:=\tau_{\theta_i}^{-1}(b_{\theta_i})$. 

\begin{lemma}\label{lem:disk_excahnging}
    There is a Powell move that exchanges $a^1=h_{\theta_i}\tau_{\theta_i}(a'_i)$ and $h_{\theta_i}\tau_{\theta_i}(a'_{i+1})$. 
\end{lemma}

\begin{proof}
    Note that $h_{\theta_i}\tau_{\theta_i}(a'_{i+1})$ is disjoint from $h_{\theta_i}\tau_{\theta_i}(a'_i)$ by Lemma~\ref{lem:primitive_disks}. 
    
    First suppose that $h_{\theta_i}\tau_{\theta_i}(a'_{i+1})$ is disjoint from $b^1$. Then, by induction there exists a Powell move that carries $h_{\theta_i}\tau_{\theta_i}(a'_{i+1})$ to $a^g$. Since $D_\eta$ carries $a^g$ to $a^1$, the conclusion follows. 
    
    Next suppose that $h_{\theta_i}\tau_{\theta_i}(a'_{i+1})$ intersects $b^1=h_{\theta_i}\tau_{\theta_i}(b'_i)$ in a single point. It follows from Proposition~3.5 in \cite{Freedman-Scharlemann} (by switching the roles of $A$ and $B$) that there exists a Powell move that exchanges $a^1$ and $h_{\theta_i}\tau_{\theta_i}(a'_{i+1})$ with $b^1$ left invariant. 
\end{proof}

\begin{lemma}\label{lem:P-equivalence2}
    For each $i$, $h_{\theta_i}\tau_{\theta_i}$ and $h_{\theta_{i+1}}\tau_{\theta_{i+1}}$ are Powell equivalent. 
\end{lemma}

\begin{proof}
    By Lemma~\ref{lem:disk_excahnging}, $h_{\theta_i}\tau_{\theta_i}$ is Powell equivalent to a map $k:(S^3,T) \to (S^3,T)$ that carries $a'_{i+1}$ to $a^1$. On the other hand, by Lemma~\ref{lem:P-equivalence} $k$ is Powell equivalent to $h_{\theta_{i+1}}\tau_{\theta_{i+1}}$. Thus, the conclusion follows. 
\end{proof}

It follows from Lemma~\ref{lem:P-equivalence2} that $h_0 \tau_0$ and $h_{2\pi}\tau_{2\pi}$ are Powell equivalent. Recall that $h_0=h_{2\pi}=\id_{S^3}$. Since $\tau_0=\id_{S^3}$ by definition, $\tau_{2\pi}$ is Powell equivalent to $\id_{S^3}$. This completes the proof of Theorem~\ref{thm:Powell}. 

\bibliographystyle{amsalpha} 
\bibliography{reference}

\end{document}